\documentclass[12pt]{amsart}
\usepackage{amssymb}
\usepackage{mathtools}
\usepackage[english]{babel}
\usepackage[utf8]{inputenc}
\usepackage{epsfig}
\usepackage{tikz}
\usepackage{mathptmx}
\usepackage{amsmath,amsfonts,amssymb,amsthm}
\usepackage{mathrsfs}
\usepackage{enumitem}
\usepackage[all]{xy}
\usepackage{graphicx}
\usepackage{latexsym}
\usepackage{verbatim}
\usepackage{hyperref}
\numberwithin{equation}{section}

\def\Re{{\sf Re}\,}
\def\Im{{\sf Im}\,}

\newcommand{\D}{\mathbb D}

\newcommand{\de}{\partial}
\newcommand{\R}{\mathbb R}
\newcommand{\Ha}{\mathbb H}

\newcommand{\C}{\mathbb C}

\newcommand{\Hol}{{\sf Hol}}
\newcommand{\Aut}{{\sf Aut}(\mathbb D)}

\newcommand{\strip}{\mathbb{S}}

\def\Aut{{\sf Aut}}

\def\Re{{\sf Re}\,}
\def\Im{{\sf Im}\,}

\def\Re{{\sf Re}\,}
\def\Im{{\sf Im}\,}

\def\1#1{\overline{#1}}
\def\2#1{\widetilde{#1}}
\def\3#1{\widehat{#1}}
\def\4#1{\mathbb{#1}}
\def\5#1{\frak{#1}}
\def\6#1{{\mathcal{#1}}}

\def\Re{{\sf Re}\,}
\def\Im{{\sf Im}\,}

\newcommand{\mcite}[1]{\csname b@#1\endcsname}

\theoremstyle{plain}

\def\Aut{{\sf Aut}}

\def\Re{{\sf Re}\,}
\def\Im{{\sf Im}\,}

\newtheorem{theorem}{Theorem}[section]

\newtheorem{proposition}[theorem]{Proposition}
\newtheorem{corollary}[theorem]{Corollary}

\theoremstyle{definition}
\newtheorem{definition}[theorem]{Definition}

\theoremstyle{remark}
\newtheorem{remark}[theorem]{Remark}

\numberwithin{equation}{section}

\title[Asymptotic upper bound for tangential speed of parabolic semigroups]{Asymptotic upper bound for tangential speed of parabolic semigroups of holomorphic self-maps in the unit disc}

\author[D. Cordella]{Davide Cordella$^\dag$
}	
\address{D. Cordella: Dipartimento di Matematica, Universit\`a di Roma ``Tor Vergata", Via della Ricerca
Scientifica 1, 00133, Roma, Italia.} 
\email{cordella@mat.uniroma2.it}

\thanks{$^\dag$Partially supported by PRIN {\sl Real and Complex
Manifolds: Topology, Geometry and holomorphic dynamics} n.2017JZ2SW5 and  by the MIUR Excellence Department Project awarded to the
Department of Mathematics, University of Rome Tor Vergata, CUP E83C18000100006}

\long\def\ReM#1{\relax}

\begin{document}
	
\selectlanguage{english}

\begin{abstract}
We show that the tangential speed $v^T_\phi(t)$ of a parabolic semigroup $(\phi_t)$ of holomorphic self-maps in the unit disc is asymptotically bounded from above by $(1/2)\log t$, proving a conjecture by Bracci. In order to show the proof we need a result of ``asymptotical monotonicity'' of the tangential speed for proper pairs of parabolic semigroups with positive hyperbolic step.
\end{abstract}

\maketitle

\tableofcontents

\section[Introduction]{Introduction}
Semigroups of holomorphic self-maps in the unit disc $\D:=\{\zeta\in \C: |\zeta|<1\}$ have been, and still are, object of several studies since more than one century, both for their intrinsic interest and for applications. See for instance \cite[Section 1.4]{Ababook89}, and the monographs \cite{BCDbook,EliShobook10,Shobook01}.

In this paper, we focus   on \emph{non-elliptic} semigroups $(\phi_t)_{t\ge 0}$, namely, those semigroups  for which all the orbits converge (uniformly on compact subsets) to a single point $\tau$ of the unit circle, the \emph{Denjoy-Wolff point} of the semigroup, and we study the rate of  convergence to such a point. 

In order to describe the rate of  convergence of a non-elliptic continuous semigroup $(\phi_t)$ of holomorphic self-maps of the unit disc---a semigroup in $\D$ for short---with respect to the parameter of the semigroup, in the paper \cite{Br} (see also \cite[Chapter 16]{BCDbook}) Bracci introduced three quantities, called ``speeds''. The first one, the \emph{total speed} $v_\phi(t)$ is simply the hyperbolic distance from the origin $0$ to $\phi_t(0)$. If $\tau\in\partial \D$ is the Denjoy-Wolff point of $(\phi_t)$, and $\gamma_\tau$ is the diameter from $-\tau$ to $\tau$, there exists a unique point $\xi_t\in \gamma_\tau$ so that the hyperbolic distance between $\phi_t(0)$ and $\xi_t$ equals the hyperbolic distance between $\phi_t(0)$ and $\gamma_\tau$. The \emph{orthogonal speed} $v^o_\phi(t)$ of $(\phi_t)$ is the hyperbolic distance between $0$ and $\xi_t$, while, the \emph{tangential speed} $v^T_\phi(t)$ is the hyperbolic distance between $\xi_t$ and $\phi_t(0)$. 

By the properties of Poincar\'e metric, the total speed is the sum of the two components, up to  universal additive constants.  Moreover, $v^T_\phi(t)\leq v^o_\phi(t)+4\log 2$ for all $t\ge 0$. The orthogonal component is related to the \emph{rate of convergence} of the semigroup as $v^T_\phi(t)$ goes like $-\frac{1}{2}\log|\tau-\phi_t(0)|$ for $t\to+\infty$, whereas the tangential component gives  information about the \emph{slope} of convergence: the tangential speed is bounded from above if and only if the orbit of the semigroup converges non-tangentially to $\tau$.

It is known that $\lim_{t\to+\infty}v_\phi(t)/t=\lambda/2$ where $\lambda\ge 0$ is the so-called \emph{spectral value} of the (non-elliptic) semigroup (see \cite[Proposition 16.2.1]{BCDbook}). In the case of hyperbolic semigroups $(\lambda>0)$ we have always non-tangential convergence to the Denjoy-Wolff point, so $v^T_\phi(t)$ is bounded and the orthogonal speed $v^o_\phi(t)$ goes like $v_\phi(t)$. Moreover, as shown in \cite[Proposition 16.2.2]{BCDbook},
\[
\liminf_{t\to+\infty}\,\left[v_\phi(t)-\frac{\lambda}{2}t\right]>-\infty
\]
and this asymptotic lower bound is sharp by taking a group of hyperbolic automorphisms. The same holds for the orthogonal speed, and then it follows that $|\tau-\phi_t(0)|\le Ke^{-\lambda t}$ for a certain positive constant $K$.

In the parabolic case $\lambda=0$ and so $\lim_{t\to+\infty}v_\phi(t)/t= 0$. This gives an asymptotic upper bound that cannot be improved, in the sense that for any function $f(t)\ge 0$ such that $\lim_{t\to+\infty}f(t)/t= 0$ there exists a parabolic semigroup $(\phi^f_t)$ for which $v_{\phi^f}(t)/f(t)$ is not bounded from above (see \cite[Proposition 16.2.8]{BCDbook}). As for a lower bound, in the worst case total speed grows as a logarithm. More precisely
\[
\liminf_{t\to+\infty}\,\left[v_\phi(t)-\frac{1}{4}\log t\right]>-\infty
\]
where $1/4$ is the best possible constant: we still refer to \cite[Proposition 16.2.2]{BCDbook}. The same bound was obtained by Betsakos in \cite{Bet} replacing the total speed with the orthogonal one. This result has been improved by Betsakos, Contreras, and D\'iaz-Madrigal in \cite{BetCD}, who gave other estimates according to the image of the Koenigs function of the semigroup. In \cite{BCK}, the author together with Bracci and Kourou partially proved a conjecture of Bracci \cite{Br} showing that in most cases the orthogonal speed is ``asymptotically monotone'' with respect to the image of the Koenigs function, which explain Betsakos' result from a geometrical point of view.

For a parabolic semigroup, either all orbits  converge non-tangentially to the Denjoy-Wolff point, or no one does; in the latter case $v^T_\phi(t)$ is unbounded. But, how much can it grow? For a group of parabolic automorphisms of the unit disc, the tangential speed behaves asymptotically like $(1/2)\log t$, up to an additive constant. The same result holds more generally for all semigroups of the form $\phi_t(z)=F^{-1}(F(z)+it)$ where $F$ is a conformal mapping (in fact, the Koenigs function of the semigroup) sending the unit disc $\D$ into a sector $W_\theta=i\{\zeta\in\C:-\theta<\arg\zeta<0\}$ for some $\theta\in(0,\pi]$ (see \cite[Corollary 16.2.6]{BCDbook}).

From these computations, Bracci conjectured in \cite{Br} that $(1/2)\log t$ gives an upper bound for all parabolic semigroups. The main result of this paper is to show that the conjecture is indeed true:
\begin{theorem}\label{thm:main}
	Let $(\phi_t)$ be a parabolic semigroup in $\D$. Then
	\begin{equation}\label{eq:conjecture}
		\limsup_{t\to+\infty}\left[v^T_\phi(t)-\frac{1}{2}\log t\right]<+\infty.
	\end{equation}
\end{theorem}

The idea of the proof is first to reduce the study to semigroups for which the associated Koenigs functions map onto Jordan domains  (see Proposition \ref{prop: horo_redux}). Next, using quasi-geodesics and properties of Gromov hyperbolic spaces, we show that, if $(\phi_t)$ is a parabolic semigroup of zero hyperbolic step, its tangential speed is always less than or equal to the maximum of the tangential speeds of the two parabolic semigroups with positive hyperbolic steps which are naturally defined by dividing the image of the Koenigs function of $(\phi_t)$ with a vertical line (see Proposition~\ref{prop:positive_step_red}). Finally, we prove that Theorem~\ref{thm:main} holds for parabolic semigroups with positive hyperbolic steps by showing that, for some semigroups in this class, the tangential speed is ``asymptotically monotone'' with respect to the image of Koenigs functions (see Proposition~\ref{prop: vt_mono}). 

The outline of the paper is the following. In Section \ref{sect:prelim} there are some preliminary results about hyperbolic metric in simply connected domains and  semigroups in the unit disc, with a focus on the associated Koenigs functions (and their images, the Koenigs domains) and the way they are related to the type of a semigroup and the convergence of orbits. 
In Section \ref{sect:speeds} we recall some facts about the speeds of convergence and we prove Proposition~\ref{prop: horo_redux}.
The  proof of Theorem \ref{thm:main} is presented in Section \ref{sect:bound}, together with the proof of  Proposition~\ref{prop: vt_mono}. For this, we need some  harmonic measure theory, and we refer, for instance, to \cite[Chapter 7]{BCDbook} and the monograph \cite{GarMar}.

Finally, in Section~\ref{sect:appl} we apply Proposition~\ref{prop: vt_mono} in order to get estimates for the tangential and orthogonal speeds of semigroups whose Koenigs functions have images equal to a ``half-parabola'' domain $\Pi_{\alpha,m}=\{z\in\C:\Re z>0, \Im z>m(\Re z)^{\alpha}\}$ where $\alpha>1$ and $m>0$.

\medskip
I would like to thank the referee for all the remarks and comments made.

\section{Preliminaries}\label{sect:prelim}
\subsection{Hyperbolic distance in simply connected domains}
Let $\D:=\{z\in\C:|z|<1\}$ be the open unit disc in the complex plane $\C$. The \emph{hyperbolic norm} in $\D$ is defined as
\[
\varkappa_\D(z;\nu):=\frac{|\nu|}{1-|z|^2}\quad z\in\D,\,\nu\in T_z\D\cong \C,
\]
so for any $\gamma:[0,1]\to\D$ Lipschitz continuous path we can assign a \emph{(hyperbolic) length}
\[
\ell_\D(\gamma):=\int_0^1\varkappa_\D(\gamma(t);\gamma'(t))\,dt.
\]
Then, for any pair of points $z,w\in \D$, the \emph{hyperbolic} (or \emph{Poincar\'e}) \emph{distance} is the integrated distance
\[
k_\D(z,w):=\inf_{\gamma\in\Gamma_{z,w}}\ell_\D(\gamma),
\]
where $\Gamma_{z,w}:=\{\gamma:[0,1]\to\D\text{ Lipschitz}:\gamma(0)=z,\,\gamma(1)=w\}$. There is an explicit formula (see \cite[Theorem 1.3.5]{BCDbook}):
\[
k_\D(z,w):=\frac{1}{2}\log\frac{1+|\Theta_w(z)|}{1-|\Theta_w(z)|},\quad z,w\in\D,
\] where $\Theta_w$ is the automorphism of $\D$ given by $z\mapsto (w-z)/(1-\overline{w} z)$.

Now take any open simply connected set $\Omega\subsetneq\C$. By the Riemann Mapping Theorem there exists a conformal equivalence $f:\D\to\Omega$. Hyperbolic metric can be transfered to $\Omega$ by taking the pull-back under the map $f^{-1}$. The hyperbolic norm becomes
\[
\varkappa_\Omega(z;\nu):=\varkappa_\D(f^{-1}(z);(f'(f^{-1}(z)))^{-1}\nu),\quad z\in\Omega,\,\nu\in T_z\Omega\cong \C,
\]
while the hyperbolic length $\ell_\Omega$ and distance $k_\Omega$ are defined as above. The classical Schwarz-Pick Lemma can be restated in the language of hyperbolic metrics by saying that hyperbolic norms and distances are decreasing under holomorphic maps and invariant under biholomorphisms.

A \emph{geodesic} is a smooth curve for which the hyperbolic length between any pair of points coincides with their hyperbolic distance. Geodesics for $k_\D$ are well-know: they are given by diameters and arcs of circles intersecting $\de\D$ orthogonally. In particular there exist a unique geodesic between two points in $\D$, and a unique geodesic ray starting from a point in $\D$ and converging to a fixed point in $\de\D$.

For a generic simply connected domain $\Omega\subsetneq\C$, since any Riemann map $f:\D\to\Omega$ is an isometry between the metric space $(\D,k_\D)$ and $(\Omega,k_\Omega)$, we still have uniqueness of the geodesic between two points of $\Omega$. As for boundary issues, in general one can consider the Carath\'eodory boundary $\de_C\Omega$ of the prime ends and say there exists a unique geodesic ray arising from a inner point of $\Omega$ and converging to a fixed prime end (see for instance \cite[Chapter 4]{BCDbook}).

One particular case that will be important in the next sections is the right half-plane $\Ha:=\{w\in\C:\Re w>0\}$. Fix $\tau\in\de\D$: a Riemann map is given by the \emph{Cayley transform} 
\begin{equation}\label{eq:Cayley}
C_\tau:\D\to \Ha,\quad
C_\tau(z):=\frac{\tau+z}{\tau-z}.
\end{equation}
Using such a map one can get the expressions for hyperbolic norm and distance:
\begin{align*}
\varkappa_\Ha(w;\nu)&=\frac{|\nu|}{2\Re w},\quad w\in\Ha,\nu\in T_w\Ha\cong\C \\
k_\Ha(w_1,w_2)&=\frac{1}{2}\log\frac{1+\left|\frac{w_1-w_2}{w_1+\overline{w_2}}\right|}{1-\left|\frac{w_1-w_2}{w_1+\overline{w_2}}\right|},\quad w_1,w_2\in\Ha.
\end{align*}
The map $C_\tau$ is a M\"obius transformation, so the geodesics in $\Ha$ with the hyperbolic metric are given by arcs of circles crossing the imaginary axis orthogonally and lines parallel to the real axis.

\subsection{Semigroups in the unit disc}
A \emph{continuous semigroup of holomorphic self-maps in the unit disc} $(\phi_t)_{t\ge 0}$ (shortly, a \emph{semigroup in $\D$}) is a family of holomorphic functions $\phi_t\in\Hol(\D,\D)$ such that the map $t\mapsto\phi_t$ from $\R^+:=[0,+\infty)$ to $\Hol(\D,\D)$ has the following properties
\begin{enumerate}[label=(\roman*)]
		\item it is a semigroup homomorphism from $(\R^+,+)$ to $(\Hol(\D,\D),\circ)$;
		\item it is continuous, taking the Euclidean topology in $\R^+$ and the topology of uniform convergence on compact subsets in $\Hol(\D,\D)$.
\end{enumerate}
We say that a semigroup in $\D$ is \emph{non-elliptic} if the maps $\phi_t$ have no fixed point in $\D$. In that case, there exists a point $\tau\in\de\D$, the \emph{Denjoy-Wolff point} of the semigroup, such that $\lim_{t\to+\infty}\phi_t(z)=\tau$ for any $z\in\D$ and the convergence in uniform on any compact subset $K\subset\D$. This result, combined with Julia's Lemma \cite[Theorem 1.4.7]{BCDbook} says that for all $t>0$ and $R>0$
\begin{equation}\label{eq: horocy}
	\phi_t(\mathscr{E}(\tau,R))\subset\mathscr{E}(\tau,R)
\end{equation}
where $\mathscr{E}(\tau,R):=\{z\in\D:|\tau-z|^2<R(1-|z|^2)\}$ is the \emph{horocycle} of (hyperbolic) radius $R$ centered in $\tau$, which is in Euclidean terms an open circle contained entirely in $\D$ of radius $R/(R+1)$ and tangent to $\de\D$ at $\tau$.
Moreover, it can be proven that for any $t\ge 0$ the angular derivative at the Denjoy-Wolff point exists and is given by $\phi_t'(\tau)=e^{-\lambda t}$, where $\lambda\ge 0$. We say that a non-elliptic semigroup in $\D$ $(\phi_t)$ is \emph{parabolic} when $\lambda=0$, otherwise it is \emph{hyperbolic} of \emph{spectral value} $\lambda>0$.

For parabolic semigroups we have a further classification. Let us define the \emph{$1$-hyperbolic step} $s_1(\phi_t,z):=\lim_{t\to+\infty}k_\D(\phi_t(z),\phi_{t+1}(z))$. Two cases may occur: $s_1(\phi_t,z)$ is positive for all $z\in\D$ or it is always zero. We say that a parabolic semigroup has \emph{positive hyperbolic step} in the first situation and has \emph{zero hyperbolic step} in the second one.

\subsection{Canonical models}
For any non-elliptic semigroup $(\phi_t)$ in $\D$ one can associate in an essentially unique way a \emph{holomorphic model} of the form $(\Omega^*,h,z\mapsto z+it)$. 

Here $h:\D\to\C$ is a univalent map, so that its image $\Omega:=h(\D)$ is a simply connected domain in $\C$ which is \emph{starlike at infinity} in the positive direction of imaginary axis \footnote{In the following we will only write \emph{starlike at infinity} assuming that we have chosen this direction.} (which means that for all $t\ge 0$, $\Omega+it\subset \Omega$); furthermore $h$ intertwines $\phi_t$ with $z\mapsto z+it$, i.e. $h(\phi_t(z))=h(z)+it$ for all $z\in\D$ and $t\ge 0$. Note that this means that the orbits are mapped to vertical half-lines. The set $\Omega^*$ is defined as $\Omega^*:=\bigcup_{t\ge 0}(\Omega-it)$.  The model detects the type of our semigroup in the following way:
\begin{enumerate}[label=(\alph*)]
	\item $(\phi_t)$ is hyperbolic of spectral value $\lambda>0$ if it has a holomorphic model $(\Omega^*,h,z\mapsto z+it)$ so that $\Omega^*=\strip_{0,\pi/\lambda}:=\{w\in\C: 0<\Re w<\pi/\lambda\}$;
	\item $(\phi_t)$ is parabolic of positive hyperbolic step when it admits a holomorphic model $(\Omega^*,h,z\mapsto z+it)$ such that $\Omega^*=\Ha$ or $\Omega^*=-\Ha$;
	\item $(\phi_t)$ is parabolic of zero hyperbolic step if there exists a holomorphic model $(\Omega^*,h,z\mapsto z+it)$ so that $\Omega^*=\C$.
\end{enumerate}
As in \cite{BCDbook}, we will call $h$ the \emph{Koenigs function} of $(\phi_t)$ and $\Omega=h(\D)$ its \emph{Koenigs domain}. These are well defined up to composing with translations.

\subsection{Koenigs domains and convergence}
The Euclidean geometry of the Koenigs domain associated to a non-elliptic semigroup $(\phi_t)$ in $\D$ gives information about the convergence of orbits to the Denjoy-Wolff point.

Let $\tau\in\de\D$ be the Denjoy-Wolff point of $(\phi_t)$. The \emph{slope} of $\phi_t$ at $z\in\D$ is defined as the cluster set of $\arg(1-\overline{\tau}\phi_t(z))$ for $t\to+\infty$, which is a closed connected interval contained in $[-\pi/2,\pi/2]$, possibly reduced to a point. If it is contained in $(-\pi/2,\pi/2)$, the orbit $t\mapsto\phi_t(z)$ is said to converge \emph{non-tangentially} to $\tau$. Equivalently, $t\mapsto\phi_t(z)$ has non-tangential convergence if there exists $R>1$ such that the orbit is eventually contained in the \emph{Stolz region}
\begin{equation}\label{eq:stolzdef}
\mathcal{S}(\tau,R):=\{z\in\D:|\tau-z|<R(1-|z|)\}.
\end{equation}
An orbit $t\mapsto\phi_t(z)$ converges \emph{tangentially} to $\tau$ if the slope of the semigroup at $z$ reduces to $\{-\pi/2\}$ or $\{\pi/2\}$.

Now let $\Omega=h(\D)$ be the Koenigs domain of $(\phi_t)$ and fix $p\in\Omega$. Consider the following quantities:
\begin{align}\label{eq:distances}
\begin{aligned}
	\delta^+_p(t)&:=\min\,\{t,\inf\,\{|\zeta-p-it|:\zeta\in\de\Omega,\Re\zeta\ge\Re p\}\},\\
\delta^-_p(t)&:=\min\,\{t,\inf\,\{|\zeta-p-it|:\zeta\in\de\Omega,\Re\zeta\le\Re p\}\}.
\end{aligned}
\end{align}
The following deep result holds.
\begin{theorem}[{\cite[Theorem 1.2 and 1.3]{BCDGZ}}]\label{thm:quasi-sym-nontg}
	Let $(\phi_t)$ be a non-elliptic semigroup in $\D$ with Denjoy-Wolff point $\tau\in\de\D$. Let $h:\D\to \C$ be its Koenigs map and $\Omega=h(\D)$ the relative Koenigs domain. 
	\begin{enumerate}[label=(\alph*)]
		\item For some (and hence any) $z\in\D$ the orbit $t\mapsto\phi_t(z)$ converges \emph{non-tangentially} to $\tau$ if and only if for some (and hence any) $p\in\Omega$ there is a pair of constants $0<c<C$ such that for all $t\ge 0$ \[
		c\delta^+_p(t)\le\delta^-_p(t)\le C\delta^+_p(t).
		\]
		\item For some (and hence any) $z\in\D$ the orbit $t\mapsto\phi_t(z)$ converges \emph{tangentially} to $\tau$  if and only if one of the following holds:
		\begin{itemize}
			\item $\lim_{t\to+\infty}\delta^+_p(t)/\delta^-_p(t)=+\infty$ for some (and hence any) $p\in\Omega$;
			\item $\lim_{t\to+\infty}\delta^+_p(t)/\delta^-_p(t)=0$ for some (and hence any) $p\in\Omega$.
		\end{itemize}
	In the first case, the slope at any point in $\D$ is $\{-\pi/2\}$, while in the second one we have slope $\{\pi/2\}$ at all points of the unit disc.
	\end{enumerate}
\end{theorem}
\subsection{Quasi-geodesics}
The proof given in \cite{BCDGZ} of Theorem \ref{thm:quasi-sym-nontg} is based on the construction of a \emph{quasi-geodesic} of $\Omega$ for the hyperbolic metric $k_\Omega$, i.e. a Lipschitz curve $\sigma:[0,+\infty)\to\Omega$ such that $k_\Omega(\sigma(0),\sigma(t))\to+\infty$ for $t\to+\infty$ and there exist $A\ge1$ and $B\ge0$ for which
\[
\ell_\Omega(\sigma\mid_{[s,t]})\le Ak_\Omega(\sigma(s),\sigma(t))+B\quad\text{ for all }0\le s<t<+\infty.
\]
A quasi-geodesic satisfying this property for a fixed pair $(A,B)$ is called a \emph{$(A,B)$-quasi-geodesic}.
\begin{proposition}[{\cite[Theorem 4.2]{BCDGZ}}]
	Let $\Omega\subsetneq\C$ be a simply connected domain which is starlike at infinity. Fix $p\in\Omega$. Let $\delta^+_p,\delta^-_p$ be defined as in \eqref{eq:distances}. Then the curve $\sigma:[0,+\infty)\to\Omega$
	\begin{equation}\label{eq:sigmadefin}
	\sigma(t):=\frac{1}{2}(\delta^+_p(t)-\delta^-_p(t))+i(\Im p+t)
	\end{equation}
is a quasi-geodesic for the hyperbolic distance in $\Omega$.
\end{proposition}

In a general simply connected domain we do not know the exact shape geodesics because there is no explicit expression of the Riemann map. However we only need proper estimates of the metric to get quasi-geodesics and it turns out that they are close to geodesics by the following result.
\begin{proposition}[Shadowing Lemma]\label{prop:shad}
	Let $\Omega\subsetneq\C$ be a simply connected domain which is starlike at infinity. Let $A\ge1$, $B\ge0$. Then there exists $M=M(A,B,\Omega)>0$ such that for any $(A,B)$-quasi-geodesic $\sigma:[0,+\infty)\to\Omega$ there exist a $k_\Omega$-geodesic $\eta:[0,+\infty)\to\Omega$ such that $\sigma(0)=\eta(0)$ and for any $t\ge 0$ it is
	\[
	\inf_{s\ge0} k_\Omega(\sigma(t),\eta(s))<M,\quad\inf_{s\ge0} k_\Omega(\eta(t),\sigma(s))<M.
	\]
	Furthermore, $\sigma$ and $\eta$ converge to the same prime end in the Carath\'eodory boundary $\de_C\Omega$.
\end{proposition}
A direct proof can be found in \cite[Theorem 6.3.8]{BCDbook}. This is indeed a general fact about metric spaces which are \emph{Gromov hyperbolic}: for instance see \cite[p. 41]{CDP}.
\section{Speeds of convergence}\label{sect:speeds}
The \emph{speeds of convergence} were first introduced in the paper \cite{Br}: these are functions of the variable $t\in[0,+\infty)$ which are associated to a non-elliptic semigroup in $\D$ and which asymptotic behavior give further information about convergence to the Denjoy-Wolff point.
\begin{remark}[Notation]
	Let $f,g:(0,+\infty)\to\R$. In the following $f(t)\lesssim g(t)$ will be a shortcut for $\limsup_{t\to+\infty}[f(t)-g(t)]<+\infty$ and we will write $f(t)\sim g(t)$ whenever $f(t)\lesssim g(t)$ and $g(t)\lesssim f(t)$, or equivalently when $\sup_{t\ge a}|f(t)-g(t)|$ is bounded for any $a>0$.
\end{remark}
\begin{definition}
	Let $(\phi_t)$ be a non-elliptic semigroup in $\D$ and $\tau\in \de\D$ its Denjoy-Wolff point. For $t\ge 0$, the \emph{total speed} of the semigroup is simply given by the hyperbolic distance
		\[
		v_\phi(t):=k_\D(0,\phi_t(0)).
		\]
	The diameter $\gamma_\tau$ from $-\tau$ to $\tau$ is a $k_\D$-geodesic line. For any non-negative $t$, let $\xi_t=\pi_{\gamma_\tau}(\phi_t(0))$ be the hyperbolic projection of $\phi_t(0)$ onto $\gamma_\tau$, that is to say the unique point along the radius minimizing the hyperbolic distance from $\phi_t(0)$ which is given by the intersection of $\gamma_\tau$ with the $k_\D$-geodesic passing through $\phi_t(0)$ and crossing $\gamma_\tau$ orthogonally (\cite[Proposition 6.5.2]{BCDbook}). The \emph{orthogonal speed} of $(\phi_t)$ at $t$ is
		\[
		v^o_\phi(t):=k_\D(0,\xi_t)
		\]
	and the \emph{tangential speed} is
		\[
		v^T_\phi(t):=k_\D(\phi_t(0),\xi_t).
		\]
\end{definition}
A kind of ``Pythagoras' Theorem'' holds without squaring the terms, up to additive constants.
\begin{proposition}[{\cite[Proposition 3.4]{Br}}]
	Let $(\phi_t)$ be a non-elliptic semigroup in $\D$. Then for all $t\ge 0$ it is $v_\phi^o(t)+v_\phi^T(t)\sim v(t)$. More precisely,
	\begin{equation}\label{eq: Pythagoras}
		v_\phi^o(t)+v_\phi^T(t)-\frac{1}{2}\log 2\le v_\phi(t)\le v_\phi^o(t)+v_\phi^T(t).
	\end{equation}
\end{proposition}
By \eqref{eq: Pythagoras} and Julia's Lemma we can also deduce (see \cite[Proposition 5.4]{Br}) that tangential speed is bounded by the orthogonal speed, up to a constant: for every $t\ge 0$
\[
v_\phi^T(t)\le v^o_\phi(t)+4\log 2.
\]
\begin{remark}
	Of course a choice has been made in defining the speeds of convergence by setting the initial point at the origin. Analogous definitions may be given for a different choice of the initial point, so the speeds would change but not the asymptotic behavior, as the respective differences are in fact bounded by a constant (see \cite[Lemma 3.7]{Br}). 
\end{remark}

The Euclidean meaning of speeds of convergence for a non-elliptic semigroup in $\D$ is given by the following proposition.
\begin{proposition}[{\cite[Proposition 3.8]{Br}}]
	Let $(\phi_t)$ be a non-elliptic semigroup in $\D$ and $\tau\in\de\D$ its Denjoy-Wolff point. Then for all $t\ge 0$
	\begin{align}
	\left|v_\phi(t)-\frac{1}{2}\log\frac{1}{1-|\phi_t(0)|}\right|&\le\frac{1}{2}\log 2\label{eq:v_eu},\\
	\left|v^o_\phi(t)-\frac{1}{2}\log\frac{1}{|\tau-\phi_t(0)|}\right|&\le\frac{1}{2}\log 2\label{eq:vo_eu},\\
	\left|v^T_\phi(t)-\frac{1}{2}\log\frac{|\tau-\phi_t(0)|}{1-|\phi_t(0)|}\right|&\le\frac{3}{2}\log 2.\label{eq:vt_eu}
	\end{align}
\end{proposition}
A direct consequence of \eqref{eq:vt_eu} is that the orbits of $(\phi_t)$ have non-tangential convergence to the Denjoy-Wolff point if and only if $\limsup_{t\to+\infty}v^T_\phi(t)<+\infty$.

Geodesics may be replaced by quasi-geodesics in order to get the asymptotic behavior of tangential speed. 
\begin{proposition}\label{prop: gamma-sigma}
	Let $(\phi_t)$ be a non-elliptic semigroup in $\D$, with Koenigs map $h:\D \to \C$, Koenigs domain $h(\D)=\Omega$ and Denjoy-Wolff point $\tau\in\de\D$. Let $A\ge1$, $B\ge0$ and
	$\sigma:[0,+\infty)\to\Omega$ be a $(A,B)$-quasi-geodesic in $\Omega$ such that $h^{-1}(\sigma(s))\to\tau$ for $s\to+\infty$. There exists $M=M(A,B,\Omega)>0$ such that for any $t\ge0$
		\[
		|v_\phi^T(t)-\inf_{s\ge 0}k_{\Omega}(h(0)+it,\sigma(s))|<M.
		\]
\end{proposition}
\begin{proof}
	It can be assumed that $\sigma(0)=h(0)$: otherwise one can
	extend $\sigma$ with an initial fixed arc joining $h(0)$ to $\sigma(0)$, still having a quasi-geodesic ray. Let $\gamma\,$ be the $k_\Omega$-geodesic ray given by the image under $h$ of the radius from $0$ to $\tau$. By Proposition \ref{prop:shad} there exists a constant $M>0$ depending on the domain and the  such that for any $s\geq 0$ there exists $t_{s}$ for which
	$k_{\Omega}(\gamma(t_{s}),\sigma(s))<M$ and so
		\[
		v^{T}_\phi(t)\le k_{\Omega}(h(0)+it,\gamma(t_{s}))\le k_\Omega(\gamma(t_{s}),\sigma(s))+k_\Omega(h(0)+it,\sigma(s))<M+k_\Omega(h(0)+it,\sigma(s)).
		\]
	On the other side, using again Proposition~\ref{prop:shad}, if $\pi_{\gamma}(h(0)+it)=\gamma(t')$  , there exists
	$s'\geq 0$ such that $k_\Omega(\gamma(t'),\sigma(s'))<M$ and therefore
		\[
		v^{T}_\phi(t)=k_\Omega(h(0)+it,\gamma(t'))\ge k_\Omega(h(0)+it,\sigma(s'))-k_\Omega(\sigma(s'),\gamma(t'))>k_\Omega(h(0)+it,\sigma(s'))-M.\qedhere
		\]
\end{proof}

Using the Cayley transform $C=C_\tau$ defined in \eqref{eq:Cayley}, from a semigroup $(\phi_t)$ in $\D$ with Denjoy-Wolff point $\tau\in\de\D$ one can get a semigroup of holomorphic-self maps in $\Ha$ by setting
\begin{equation}\label{eq:semigr_ha}
	\psi_t:=C\circ\phi_t\circ C^{-1}.
\end{equation}
The Denjoy-Wolff point is now the point at infinity. We can define speeds of convergence in the same way by taking the initial point $C(0)=1$ and considering the projection of $\psi_t(1)$ onto $C(\gamma_\tau)=[1,+\infty)$ with respect to the distance $k_\Ha$: by conformal invariance the speeds coincide with the ones of $\phi_t$. The inclusion $\eqref{eq: horocy}$ and the Cayley transform imply that the orbit $[0,\infty)\ni t\mapsto\psi_t(1)$ lies entirely in $\{w:\Re w\ge 1\}$ and $t\mapsto \Re\psi_t(1)$ is non-decreasing. Writing down in polar coordinates $\psi_t(1)=\rho_te^{i\theta_t}$, this means that $\rho_t$ goes to infinity and the function $t\mapsto \rho_t\cos\theta_t$ is non-decreasing. We can use some known properties of the hyperbolic distance $k_\Ha$, recollected in \cite[Lemma 5.4.1]{BCDbook}, in order to compute speeds in terms of polar coordinates. Indeed, since the point in $[1,+\infty)$ minimizing the $k_\Ha$-distance from $\rho_te^{i\theta_t}$ is $\rho_t$,
\begin{align}\label{eq:compspeed}
	\begin{split}
	v^o_\psi(t)&=k_\Ha(1,\rho_t)=\frac{1}{2}\log\rho_t,\\
	v^T_\psi(t)&=k_\Ha(\rho_te^{i\theta_t},\rho_t)\sim\frac{1}{2}\log\frac{1}{\cos\theta_t}.
	\end{split}
\end{align}

We now show that restricting to Koenigs domains with smooth boundary is not a loss of generality if we deal with asymptotic behavior of orthogonal or tangential speeds.
\begin{proposition}[Horocycle reduction]\label{prop: horo_redux}
	Let $(\phi_t)$ be a non-elliptic semigroup in $\D$. There exists another semigroup $(\hat{\phi}_t)$ in $\D$ such that
	\begin{equation}\label{eq: horo_redux}
	v^T_\phi(t)\sim v^T_{\hat{\phi}}(t),\quad v^o_\phi(t)\sim v^o_{\hat{\phi}}(t),
	\end{equation}
 and the boundary of its Koenigs domain in $\C_\infty$ is the image of a Jordan curve (passing at $\infty$) which is smooth (analytic) in the complex plane.
\end{proposition}
\begin{proof}
	Let $h:\D\to\C$ be the Koenigs map of $(\phi_t)$ and $\Omega:=h(\D)$. Let $\tau\in\de\D$ be the Denjoy-Wolff point. By conjugating with the Cayley transform $C=C_\tau$ as in \eqref{eq:Cayley}, one gets the semigroup $\psi_t:=C\circ\phi_t\circ C^{-1}$ in the half-plane $\Ha$, with Denjoy-Wolff point at infinity. For $H:=h\circ C^{-1}$ one can define the semigroup $(\hat{\psi}_t)$ in $\Ha$  
	\[
	\hat{\psi}_t(w):=-\frac{1}{2}+H^{-1}\left(H\left(w+\frac{1}{2}\right)+it\right)\qquad \text{for }t\ge 0, w\in \Ha.
	\]
	Let $\psi_t(1):=\rho_te^{i\theta_t}$. Since $H(\psi_t(w))=H(w)+it$, then \[\hat{\psi}_t(1/2)=\rho_te^{i\theta_t}-\frac{1}{2}=\left|\rho_te^{i\theta_t}-\frac{1}{2}\right|\exp\left(i\arg\left(\rho_te^{i\theta_t}-\frac{1}{2}\right)\right)\]
	and from \eqref{eq:compspeed} one deduces that
	\[
	v^o_{\hat{\psi}}(t)\sim \frac{1}{2}\log \left|\rho_te^{i\theta_t}-\frac{1}{2}\right|,\quad v^T_{\hat{\psi}}(t)\sim\frac{1}{2}\log \frac{1}{\cos(\arg(\rho_te^{i\theta_t}-1/2))}.
	\]
	Now $v^o_{\hat{\psi}}(t)=v^o_{\hat{\phi}}(t)$ where $\hat{\phi}_t:=C^{-1}\circ \hat{\psi}_t\circ C$ is a corresponding semigroup in the unit disc, and the same holds for tangential speed. As $\rho_t-1/2\le |\rho_te^{i\theta_t}-1/2|\le \rho_t$ and $\rho_t\to+\infty$ as $t\to+\infty$, there exists $M_1>0$ such that $|v^o_\phi(t)- v^o_{\hat{\phi}}(t)|<M_1$ for any $t\ge 0$.
	For the tangential speed, recalling that $1\le \mathsf{Re}\,\psi_t(1)=\rho_t\cos\theta_t$ is non-decreasing in $t$, if $R:=\lim_{t\to+\infty}\rho_t\cos\theta_t\in[1,+\infty]$ one gets that
	\begin{align*}
	\lim_{t\to+\infty}\frac{\cos(\arg(\rho_te^{i\theta_t}-1/2))}{\cos\theta_t}&=\lim_{t\to+\infty}\frac{\rho_t}{|\rho_te^{i\theta_t}-1/2|}-\frac{1}{2|\rho_te^{i\theta_t}-1/2|\cos\theta_t}=\\
	&=\begin{cases}
	1-(2R)^{-1} & \text{if }R<+\infty\\
	1 & \text{if }R=+\infty.
	\end{cases}
	\end{align*}
	In any case it can be deduced the existence of $M_2>0$ such that $|v^T_\phi(t)- v^T_{\hat{\phi}}(t)|<M_2$ and so both relations in \eqref{eq: horo_redux} hold. 
	
	The Koenigs domain of $\hat{\phi}$ is $\hat{\Omega}=H(\{w:\Re w>1/2\})$ and $H$ is conformal on $\Ha$, hence $\partial\hat{\Omega}$ is an analytic curve.
\end{proof}

\section{Upper asymptotic bound for tangential speed}\label{sect:bound}

Proposition \ref{prop: horo_redux} gives a first important reduction in order to prove Theorem \ref{thm:main}: we can assume that the parabolic semigroup $(\phi_t)$ is associated to a Koenigs domain whose boundary is a simple Jordan curve. Now we show that it is enough to consider semigroups with positive hyperbolic step to get an upper bound for tangential speed.
\begin{proposition}[Reduction to positive hyperbolic step]\label{prop:positive_step_red}
	Let $(\phi_t)$ be a parabolic semigroup in $\D$ with Koenigs map $h:\D\to\C$ and Koenigs domain $\Omega=h(\D)$. Choose $\epsilon>0$ such that the Euclidean disc $B(h(0),2\epsilon)\subset\Omega$ and consider the following domains
		\[
		\Omega^{+}_\epsilon:=\Omega\cap\left\{ z:\mathsf{Re}\,z>\mathsf{Re }\  h(0)-\epsilon\right\}, 
		\quad
		\Omega^{-}_\epsilon:=\Omega\cap\left\{ z:\mathsf{Re}\,z<\mathsf{Re }\ h(0)+\epsilon\right\}. 
		\]
	If $(\phi_t)$ has zero hyperbolic step, then $\Omega^+_\epsilon$ and $\Omega^-_\epsilon$  are Koenigs domains of two parabolic semigroups $(\phi_t^+)$ and $(\phi_t^-)$ of positive hyperbolic step and
		\begin{equation}\label{eq: positive_step_red}
		\limsup_{t\to+\infty}\,[v^T_\phi(t)-\max\{v^T_{\phi^+}(t),v^T_{\phi^-}(t)\}]<+\infty.
		\end{equation}
	The same conclusion holds when $(\phi_t)$ has positive hyperbolic step, but in this case one of the two semigroups $(\phi_t^+), (\phi_t^-)$ is hyperbolic.
\end{proposition}
\begin{proof}
	We only show the proof for $(\phi_t)$ with zero hyperbolic step. Up to a translation, we can assume $h(0)=0$. The two domains $\Omega_\epsilon^\pm$ are simply connected, starlike at infinity and they are contained in vertical half-planes, so they are Koenigs domains of parabolic semigroups in $\D$ $(\phi_t^\pm)$ of positive hyperbolic step.
	Let $\sigma$ be the quasi-geodesic given by \eqref{eq:sigmadefin}, where the initial point $p$ is set to be $0$. In the same fashion one can consider the quasi-geodesics $\sigma^{+}$, $\sigma^{-}$ in $\Omega^+_\epsilon$ and $\Omega^-_\epsilon$. Clearly
		\[
		\Re\sigma^{+}(t)>\max\{0,\Re\sigma(t)\},
		\quad \Re\sigma^{-}(t)<\min\{0,\Re\sigma(t)\}.
		\]
	Fix $t\ge 0$: there is $s_{t}^{+}$ such that $k_{\Omega^{+}_\epsilon}(it,\sigma^{+})=k_{\Omega^{+}_\epsilon}(it,\sigma^{+}(s_{t}^{+}))$
	and $s_{t}^{-}$ such that $k_{\Omega^{-}_\epsilon}(it,\sigma^{-})=k_{\Omega^{-}_\epsilon}(it,\sigma^{-}(s_{t}^{-}))$.
	Let $L_t^+$ be the $(k_{\Omega^{+}_\epsilon})$-geodesic segment from $it$ to $\sigma^{+}(s_{t}^{+})$
	and $L_t^-$ the $(k_{\Omega^{-}_\epsilon})$-geodesic segment from $it$ to $\sigma^{-}(s_{t}^{-})$.
	Then at least one of the two paths must intersect $\sigma$ somewhere. Indeed if $\Gamma$ is the set \[\Gamma:=L_t^+\cup L_t^-\cup\left\{\sigma^-(t)\mid t\ge s_t^-\right\}\cup\left\{\sigma^+(t)\mid t\ge s_t^+\right\},\] then $\Gamma\cup\{\infty\}$ the image of a Jordan curve in the Riemann sphere $\C_\infty$: by the Jordan Curve Theorem (see for instance \cite[p. 33]{Pombook75}) $\C\setminus \Gamma$ has two (path)-connected components whose boundaries coincide with $\Gamma$. The curve $\sigma$ lies in both components and does not intersect $\Gamma$ on the sides along $\sigma^+$ and $\sigma^-$, so it must cross $\Gamma$ at a least one point of $L_t^+\cup L_t^-$.
	
	Let $\mathcal{T}^+:=\{t\ge0 : L_t^+\cap\sigma([0,+\infty))\ne\varnothing\}$. If $t\in \mathcal{T}^+$ and $q\in L^{+}_t\cap\sigma([0,+\infty))$, then
		\[
		\inf_{s\ge 0}k_{\Omega^{+}_\epsilon}(it,\sigma^{+}(s))=\ell_{\Omega^{+}_\epsilon}(L^{+}_t)\ge k_{\Omega^{+}_\epsilon}(it,q)\ge k_{\Omega}(it,q)\ge\inf_{s\ge0} k_{\Omega}(it,\sigma(s))
		\]
	and by Proposition~\ref{prop: gamma-sigma} one concludes that $\sup_{t\in\mathcal{T}^+}[v_\phi^{T}(t)-v_{\phi^+}^{T}(t)]<c_0$ for some $c_0>0$. If otherwise $t\notin\mathcal{T}^+$, as $\sigma$ must intersect $L_t^-$ at some point, using the same argument one can conclude that \eqref{eq: positive_step_red} holds true.
\end{proof}

Tangential speed has not in general any ``essential monotonicity property'', but we have this kind of result in a particular case.

\begin{proposition}\label{prop: vt_mono}
	Let $(\phi_t), (\tilde{\phi}_t)$ be two parabolic semigroups in $\D$ of positive hyperbolic step and with respective Koenigs domains $\Omega, \tilde{\Omega}$ such that, up to translations:
	\begin{enumerate}[label=(\roman*)]
		\item $\Omega\subset\tilde{\Omega}$ and both are contained in $\Ha$ (or $-\Ha$);
		\item \label{en: vtmono_2} $\partial_\infty \Omega$ and $\partial_\infty \tilde{\Omega}$ are Jordan curves, both containing $i\R^+:=i(0,+\infty)$.
	\end{enumerate}
	Under these assumptions it follows that
		\begin{equation}\label{eq:vtmono0}
		\limsup_{t\to+\infty}\,[v^T_\phi(t)-v^T_{\tilde{\phi}}(t)]<+\infty.
		\end{equation}
\end{proposition}
\begin{proof}
	Assume that $\Omega$, $\tilde{\Omega}\subset\Ha$, as the proof will be analogous when considering $-\Ha$, and that both semigroups have Denjoy-Wolff point equal to $1$. Let $h:\D\to\Omega$ and $\tilde{h}:\D\to \tilde{\Omega}$ be Koenigs maps for the given pair of semigroups, and consider the corresponding semigroups in $\Ha$ $\psi_t:=C\circ\phi_t\circ C^{-1}$ and $\tilde{\psi}_t:=C\circ\tilde{\phi}_t\circ C^{-1}$ where $C$ is the Cayley transform \eqref{eq:Cayley} for $\tau=1$.
	By Carath\'eodory Extension Theorem (see for instance \cite[Section 4.3]{BCDbook}), their linearizing maps $H=h\circ C^{-1}$ and $\tilde{H}=\tilde{h}\circ C^{-1}$ extend to homeomorphisms from $\overline{\Ha}^{\infty}$ to the closures of the Koenigs domains in $\C_\infty$. Up to conjugating $(\psi_t)$ with a proper map in $\Aut(\Ha)$ (this does not change the Koenigs domain and the asymptotic behavior of hyperbolic speeds (see \cite[Proposition 16.1.6]{BCDbook}), we may also suppose that $i\R^+=H(i\R^+)$; the same assumption can be made for $\tilde{H}$. 
	
	Let $\psi_t(1)=\rho_te^{i\theta_t}$ where $\theta_t\in(-\pi/2,\pi/2)$. Since the orbits of $(\phi_t)$ have tangential convergence to the Denjoy-Wolff point with slope $-\pi/2$ (see \cite[Theorem 17.5.1]{BCDbook}), then $\theta_t\to\pi/2$ for $t\to+\infty$ and so $\theta_t>0$ for $t$ big enough.
	
	We can rewrite this in terms of harmonic measure. We use the notation $\omega(z_0,A,\Omega_0)$ for the harmonic measure at the point $z_0$ in a simply connected domain $\Omega_0$ of the Borel set $A\subset\partial_\infty\Omega_0$, which is the value $u(z_0)$ of the solution $u$ to the Dirichlet problem with boundary datum $u\mid_{\de_\infty\Omega_0}=\chi_A$.
	In this setting, for $t$ large enough $\theta_t>0$ means that $\omega(\psi_t(1),i\R^+,\Ha)>1/2$ and
		\[
		\cos\theta_t=\cos\left(\pi\omega(\psi_t(1),i\R^+,\Ha)-\frac{\pi}{2}\right)=\sin(\pi\omega(\psi_t(1),i\R^+,\Ha)).
		\]
	For the first equality, see for instance \cite[Example 7.2.6]{BCDbook}. By conformal invariance of harmonic measure \cite[Proposition 7.2.3 (3)]{BCDbook}, if $p:=H(1)$ then $\omega(\psi_t(1),i\R^+,\Ha)=\omega(p+it,i\R^+,\Omega)$ and $v^T_\phi(t)\sim -1/2\log\sin(\pi\omega(p+it,i\R^+,\Omega))$ by \eqref{eq:compspeed}. It is $\omega(p+it,i\R^+,\Omega)\le \omega(p+it,i\R^+,\tilde{\Omega}) $ by the domain monotonicity of harmonic measures  \cite[Proposition 7.2.10]{BCDbook} and then
		\begin{equation}\label{eq:vtmono1}
		\limsup_{t\to+\infty}\left[v^T_\phi(t)-\frac{1}{2}\log\frac{1}{\sin(\pi\omega(p+it,i\R^+,\tilde{\Omega}))}\right]<+\infty
		\end{equation}
	as $x\mapsto \sin x$ is decreasing in $(\pi/2,\pi)$. 
	
	Applying again conformal invariance of harmonic measures, this time using $\tilde{H}$, one gets $\omega(p+it,i\R^+,\tilde{\Omega})=\omega(\tilde{\psi}_t(q),i\R^+,\Ha)$ where $q:=\tilde{H}^{-1}(p)$. For $t$ large enough $\Im\tilde{\psi}_t(q)>0$, so
		\[
		\sin(\pi\omega(\tilde{\psi}_t(q),i\R^+,\Ha)=\cos(\arg(\tilde{\psi}_t(q))).
		\]
	A straightforward computation shows that
		\[
		\lim_{t\to+\infty}(\cos(\arg(\tilde{\psi}_t(q)-\Im q)))/(\cos(\arg(\tilde{\psi}_t(q))))=1,
		\]
	hence the tangential speed $v^T_{\tilde{\phi}}(t)$ is given by
		\[ v^T_{\tilde{\phi}}(t)\sim
		\frac{1}{2}\log\frac{1}{\cos(\arg(\tilde{\psi}_t(q)-\Im q))}\sim\frac{1}{2}\log\frac{1}{\cos(\arg(\tilde{\psi}_t(q)))}=\frac{1}{2}\log\frac{1}{\sin(\pi\omega(p+it,i\R^+,\tilde{\Omega}))}.
		\]
	Finally, by plugging this in \eqref{eq:vtmono1}, we get \eqref{eq:vtmono0}.
\end{proof}

\begin{remark}\label{rmk: vtmonomod}
	Proposition~\ref{prop: vt_mono} is also true, with analogous proof, if we replace the assumption \ref{en: vtmono_2} with
	\emph{\begin{enumerate}[label=(\roman*)]
		\item[(ii')] $\partial_\infty \Omega$ and $\partial_\infty \tilde{\Omega}$ are Jordan curves and there exists $r>0$ such that, if the domains are in $\Ha$, $\partial_\infty \Omega\cap\{w:0<\Re w<r\}=\partial_\infty \tilde{\Omega}\cap\{w:0<\Re w<r\}\ne\varnothing$, while if they are contained in $-\Ha$, $\partial_\infty \Omega\cap\{w:-r<\Re w<0\}=\partial_\infty \tilde{\Omega}\cap\{w:-r<\Re w<0\}\ne\varnothing$.
	\end{enumerate}}
\end{remark}
Now we are ready to give a proof of Theorem \ref{thm:main}.
\begin{proof}[Proof of Theorem \ref{thm:main}]
	Let $\Omega$ be the Koenigs domain of $(\phi_t)$. By Proposition~\ref{prop: horo_redux} we may assume that $\partial\Omega$ is a Jordan curve in the Riemann sphere. Using Proposition~\ref{prop:positive_step_red} we can find a couple of semigroups $(\phi_t^\pm)$ in $\D$  satisfying \eqref{eq: positive_step_red}. Moreover each one of the Koenigs domain of such semigroups is bounded by a Jordan curve in $\C_\infty$ containing a vertical half-line in the complex plane. If we make proper translations, one is contained in $\Ha$, the other one in $-\Ha$ and both boundaries contain $i\R^+$.
	
	On the other side, $\Ha$ and $-\Ha$ are trivially the Koenigs domains of the parabolic groups $\varphi^+_t(z):=C^{-1}(C(z)+it)$ and $\varphi_t^-(z):=C^{-1}(C(z)-it)$ where $C$ is any Cayley transform as in \eqref{eq:Cayley}. Since 
		\[
		v^T_{\varphi^+}(t)\sim\frac{1}{2}\log t,\quad v^T_{\varphi^-}(t)\sim\frac{1}{2}\log t,
		\]
	from Proposition~\ref{prop: vt_mono} it follows that $v^T_{\phi^\pm}(t)\lesssim\frac{1}{2}\log t$: the inequality \eqref{eq:conjecture} is a direct consequence of this and \eqref{eq: positive_step_red}.
\end{proof}

The Euclidean counterpart of Theorem \ref{thm:main} is the following:
\begin{proposition}
	Let $(\phi_t)$ be a non-elliptic semigroup in $\D$ with Denjoy-Wolff point $\tau\in\de\D$. Then there exist $R_0> 1$ so that for any $t\ge 1$
		\[
		\phi_t(0)\in \mathcal{S}(\tau,R_0\cdot t)
		\]
	where $\mathcal{S}(\tau,R_0\cdot t)$ is a Stolz region as in \eqref{eq:stolzdef}.
\end{proposition}
\begin{proof}
	By combining \eqref{eq:vt_eu} with \eqref{eq:conjecture} we infer that there exists $r_0>0$ such that for $t\ge 1$
	\[
	\log\frac{|\tau-\phi_t(0)|}{1-|\phi_t(0)|}< \log t + 3\log 2 + r_0
	\]
	and, setting $R_0=8e^{r_0}$, one concludes by taking the exponential.
\end{proof}

\begin{remark}[Essential monotonicity of orthogonal speed]
Suppose we have a pair of semigroups $\{(\phi_t),(\tilde{\phi}_t)\}$ as in Proposition \ref{prop: vt_mono} (or Remark \ref{rmk: vtmonomod}). Let $\Omega\subset\tilde{\Omega}$ be their respective Koenigs domains. By the monotonicity of hyperbolic metrics
\[
v_{\tilde{\phi}}(t)\le v_\phi(t).
\]
By combining this with \eqref{eq: Pythagoras} and \eqref{eq:vtmono0} one gets that
\begin{equation}\label{eq:vomono}
\limsup_{t\to+\infty}\,[v^o_{\tilde{\phi}}(t)-v^o_\phi(t)]<+\infty.
\end{equation}

In \cite{BCK} the same relation \eqref{eq:vomono} is obtained supposing that one of the semigroups has non-tangential convergence to the Denjoy-Wolff point, or assuming that $\tilde{\Omega}$ is a starlike domain with respect to an inner point. 

In what stated above, $\tilde{\Omega}$ has smooth boundary, it is contained in a half-plane, and it has a side in common with $\Omega$ (up to translations), but this time one can drop out the hypothesis of being a starlike domain.
\end{remark}
\section{Application: ``half-parabola'' domains}\label{sect:appl}
We want to show the asymptotic behavior of the orthogonal and tangential speeds of semigroups in $\D$ whose Koenigs domains are given by
	\[
	\Pi_{\alpha,m}:=\{z\in\C:\Re z>0, \Im z>m(\Re z)^{\alpha}\}
	\]
where $\alpha>1$ and $m>0$.

The case $\alpha=2$ is easier. Indeed the map $\Phi(z):=-iz^2+(4m)^{-1}$ is conformal from the half-strip $S:=\{z\in\C:\Im z>0,\, 0<\Re z<(2\sqrt{m})^{-1}\}$ onto the domain $\Pi_{2,m}$ (see \cite[p. 81-82]{Bie}). On the other side, $S$ is biholomorphic to $\Ha$ via $z\mapsto -i\sin(2\sqrt{m}\pi z-\pi/2)$ (see \cite[p. 276]{Neh}). Therefore we can express $k_{\Pi_\alpha}$ in terms of $k_\Ha$ and the conformal equivalence described above.

In the general case we do not know any explicit formula for a conformal mapping from $\Ha$ to $\Pi_{\alpha,m}$. But one can try to generalize the construction above with some modifications: it turns out that similar domains are obtained in this way.

Let $\mu>0$, $\alpha>1$ and denote by $\beta:=\frac{\alpha}{\alpha-1}$ its H\"older conjugate exponent. Consider the map 
	\[
	\Phi_{\alpha}(z)=i(-iz)^{\beta}
	\] 
which is well defined and holomorphic on the half-strip 
	\[
	S_{\alpha,\mu}:=\{z\in\mathbb{C}: 0<\mathsf{Re}\,z< (\mu^{1/\alpha}\beta)^{-1},\mathsf{Im}\,z> (\mu^{1/\alpha}\beta)^{-1}\cot(\pi/(2\beta))\}
	\]
by taking a branch of $z\mapsto z^{\beta}$ which is an analytic continuation of $x\mapsto x^\beta$ defined in $(0,+\infty)$.
The map $\Phi_{\alpha}$ is injective on $S_{\alpha,\mu}$ since the intersection of any circle centered in $0$ with $S_{\alpha,\mu}$ is an open arc describing an angle which amplitude is less than $\pi/2\beta<2\pi/\beta$. Therefore $\Phi_{\alpha}$ gives a conformal equivalence between $S_{\alpha,\mu}$ and the image $\Omega_{\alpha,\mu}:=\Phi_{\alpha}(S_{\alpha,\mu})$.

Let $c:=(\mu^{1/\alpha}\beta)^{-1}$ and $\eta:=\tan (\pi/(2\beta))$. The map $\Phi_{\alpha}$ extends to $\partial S_{\alpha,\mu}=L_1\cup L_2\cup L_3$ where 
	\[
	L_1=i\left(c/\eta,+\infty\right),\quad L_2=[0,c]+(c/\eta)i,\quad
	L_3=c+i\left(c/\eta,+\infty\right).
	\] 
Now $\Phi_{\alpha}(L_1)=i(c^{\beta}/\eta^{\beta},+\infty)$ is a vertical half-line. For $0\le s\le c$,
\[
\Phi_{\alpha}(s+ic/\eta)=(c^2/\eta^2+s^2)^{\frac{\beta}{2}}\left[\sin\left(\beta\arctan\frac{s\eta}{c}\right)+i\cos\left(\beta\arctan\frac{s\eta}{c}\right)\right],
\]
hence $\Phi_\alpha(L_2)$ is the image of the curve $[0,c]\ni s\mapsto\gamma_2(s):=\Phi_{\alpha}(s+ic/\eta)$ from $\gamma_2(0)=ic^{\beta}/\eta^{\beta}$ to $\gamma_2(c)=(c^2/\eta^2+c^2)^{\beta/2}$.  Note that, since $0\le\beta\arctan(s\eta/c)\le\pi/2$, then
	\[
	\frac{\partial \mathsf{Re}\,\gamma_2(s)}{\partial s}=\beta(c^2/\eta^2+s^2)^{\frac{\beta-2}{2}}\left[s \sin\left(\beta\arctan\frac{s\eta}{c}\right)+\frac{c}{\eta}\cos\left(\beta\arctan\frac{s\eta}{c}\right)\right]
	\]
is strictly greater than $0$ for any $s$, whereas
	\[
	\frac{\partial \mathsf{Im}\,\gamma_2(s)}{\partial s}=\beta(c^2/\eta^2+s^2)^{\frac{\beta-2}{2}}\left[s \cos\left(\beta\arctan\frac{s\eta}{c}\right)-\frac{c}{\eta}\sin\left(\beta\arctan\frac{s\eta}{c}\right)\right]
	\]
vanishes at $s=0$ and it is strictly negative for $0<s\le c$.

Let $t>0$ and $T:=t+c/\eta$. Thus
	\[
	\gamma_3(T):=\Phi_{\alpha}(c+iT)=(c^2+T^2)^{\frac{\beta}{2}}\left[\sin\left(\beta\arctan\frac{c}{T}\right)+i\cos\left(\beta\arctan\frac{c}{T}\right)\right].
	\]
Hence $\Phi_\alpha(L_3)$ is the image of the parametrized curve $(c/\eta,+\infty)\ni T\mapsto \gamma_3(T)$, which goes to $(c^2+c^2/\eta^2)^{\beta/2}$ for $T\to (c/\eta)^+$ and it tends to infinity when $T\to+\infty$. As $0<\beta\arctan(c/T)<\pi/2$ and $\beta>1$, it follows that
	\begin{align*}
	\frac{\partial \mathsf{Re}\,\gamma_3(T)}{\partial T}&=\beta(c^2+T^2)^{\frac{\beta-2}{2}}\left[T \sin\left(\beta\arctan\frac{c}{T}\right)-c\cos\left(\beta\arctan\frac{c}{T}\right)\right]>0,\\
	\frac{\partial \mathsf{Im}\,\gamma_3(T)}{\partial T}&=\beta(c^2+T^2)^{\frac{\beta-2}{2}}\left[T \cos\left(\beta\arctan\frac{c}{T}\right)+c\sin\left(\beta\arctan\frac{c}{T}\right)\right]>0.
	\end{align*}

Therefore $\Phi_\alpha(\partial S_{\alpha,\mu})$ is given by the concatenation of  $\Phi_\alpha(L_j), j=1,2,3$, which is a closed Jordan curve in $\C_\infty$: by Jordan Curve Theorem, it splits the Riemann sphere into two simply connected regions and the one not containing $0$ is the starlike at infinity domain $\Omega_{\alpha,\mu}$.

For the side $\gamma_3$ we have the following asymptotic behavior:
	\begin{align}\label{eq:omega_asym}
	\begin{split}
	\lim_{T\to+\infty}\frac{\mu(\mathsf{Re}\,\gamma_3(T))^{\alpha}}{\mathsf{Im}\,\gamma_3(T)}& =\lim_{T\to+\infty}\mu(c^2+T^2)^{\frac{\alpha}{2}}\sin^{\alpha-1}\left(\beta\arctan\frac{c}{T}\right)\tan\left(\beta\arctan\frac{c}{T}\right)=\\
	&=\lim_{T\to+\infty}\frac{(c^2+T^2)^{\frac{\alpha}{2}}}{T^\alpha}\mu\beta^{\alpha}c^{\alpha}=\frac{\mu\beta^{\alpha}}{\mu\beta^\alpha}=1.
	\end{split}
	\end{align}
	
Let $\zeta_0\in\Omega_{\alpha,\mu}$ lying along the geodesic line $\Phi_\alpha(S_{\alpha,\mu}\cap\{\mathsf{Re}\,z=(2\mu^{1/\alpha}\beta)^{-1}\})$ and consider the curve $t\mapsto\Phi_\alpha^{-1}(\zeta_0+it)$ for $t\ge 0$. Its real part $P(t)$ is bounded and its imaginary part $Q(t)$ is bounded from below. Let $x_0=\mathsf{Re}\,\zeta_0$ and $y_0=\mathsf{Im}\,\zeta_0$; then $\Phi_\alpha(P(t)+iQ(t))=x_0+i(t+y_0)$, i.e. 
\begin{equation}\label{eq: PQ}
(Q(t)-iP(t))^{\beta}=t+y_0-ix_0,
\end{equation}
so by taking modules in \eqref{eq: PQ} one gets that $P(t)^2+Q(t)^2=((t+y_0)^2+x_0^2)^{1/\beta}$, hence $Q(t)$ is such that
	\begin{equation}\label{eq:asympt_Q}
	\lim_{t\to+\infty} Q(t)=+\infty,\quad\lim_{t\to+\infty} \frac{Q(t)}{t^{1/\beta}}=1;
	\end{equation}
moreover $\arctan \frac{P(t)}{Q(t)}=\frac{1}{\beta}\arctan\frac{x_0}{t+y_0}$ by taking arguments in \eqref{eq: PQ}, so by easy computations one can say that
	\begin{equation}\label{eq:asympt_P}
	\lim_{t\to+\infty} P(t)=0,\quad\lim_{t\to+\infty} P(t)\cdot (\beta/x_0)t^{1/\alpha}=1.
	\end{equation}

Let $\Psi_\alpha$ be the conformal mapping $\Psi_\alpha(z)=-i\sin(-\pi/2+\pi z/c-i\pi/\eta)$ from $S_{\alpha,\mu}$ onto $\mathbb{H}$. Then 
	\[
	\Psi_\alpha(\Phi_\alpha^{-1}(\zeta_0+it))=-i\sin\left(-\frac{\pi}{2}+\frac{\pi}{c}P(z)+i\left(\frac{\pi}{c}Q(t)-\frac{\pi}{\eta}\right)\right)
	\] 
and setting $\tilde{P}(t):=\pi/2-\pi/c\cdot P(t)$, $\tilde{Q}(t):=\pi/c\cdot Q(t)-\pi/\eta$ the last term is equal to $
-i\sin (-\tilde{P}(t)+i\tilde{Q}(t))=(e^{\tilde{Q}(t)+i\tilde{P}(t)}-e^{-\tilde{Q}(t)-i\tilde{P}(t)})/2$.
Hence 
	\[
	2e^{-\tilde{Q}(t)}|\Psi_\alpha(\Phi_\alpha^{-1}(\zeta_0+it))|=\sqrt{1+e^{-4\tilde{Q}(t)}-2e^{-2\tilde{Q}(t)}\cos(2\tilde{P}(t))}
	\]
which converges to $1$ as $t\to+\infty$; from \eqref{eq:asympt_Q} one gets that \[
v^o_\phi(t)\sim\frac{1}{2}\log e^{\tilde{Q}(t)}\sim \frac{\pi}{2c}Q(t)\sim\frac{\pi}{2c} t^{1/\beta}=\frac{\pi}{2c} t^{1-1/\alpha}
\]
for any semigroup in $\D$ $(\phi_t)$ whose Koenigs domain is $\Omega_{\alpha,\mu}$ (up to translations). As for the tangential speed,
	\[
	\frac{\cos(\arg\Psi_\alpha(\Phi_\alpha^{-1}(\zeta_0+it)))}{\cos \tilde{P}(t)}=\frac{e^{\tilde{Q}(t)}-e^{-\tilde{Q}(t)}}{(e^{2\tilde{Q}(t)}+e^{-2\tilde{Q}(t)}-2\cos(2\tilde{P}(t)))^{1/2}}
	\] 
and the last term goes to $1$ as $t\to+\infty$. Clearly $\cos\tilde{P}(t)=\sin(\pi P(t)/c))$ and $P(t)\to 0$, hence \[
v^T_\phi(t)\sim\frac{1}{2}\log\frac{1}{\cos\tilde{P}(t)}\sim\frac{1}{2}\log\frac{1}{\sin(\pi P(t)/c)}\sim\frac{1}{2}\log \frac{1}{P(t)}\sim\frac{1}{2\alpha}\log t,
\]
where the last one holds because from \eqref{eq:asympt_P} it follows that $\log P(t)+(1/\alpha)\log t +\log(\beta/x_0)$ converges to $0$ for $t\to+\infty$. Eventually we have proven the following:
\begin{proposition}\label{prop:speed_omegaalpha}
	Let $\alpha>1$, $\mu>0$ and $(\varphi_t)$ a parabolic semigroup in $\D$ whose Koenigs domain is the set $\Omega_{\alpha,\mu}$ described above. Then
		\[
		v^T_{\varphi}(t)\sim\frac{1}{2\alpha}\log t,
		\]
		while for the orthogonal speed
		\[
		v^o_{\varphi}(t)\sim\frac{\mu^{1/\alpha}\alpha}{2(\alpha-1)}\pi\cdot t^{1-1/\alpha}.
		\]
\end{proposition}

Recalling the definition of $\Pi_{\alpha,m}$ and the limit \eqref{eq:omega_asym}, we infer that for $\epsilon,\epsilon'\in(0,m)$ there exist $r_1,r_2\in\R$ such that 
\[
\Omega_{\alpha,m+\epsilon}+ir_1\subset\Pi_{\alpha,m}\subset\Omega_{\alpha,m-\epsilon'}+ir_2.
\]
Combining the inclusions above, Proposition \ref{prop:speed_omegaalpha} and Proposition \ref{prop: vt_mono}, the conclusion is the following.
\begin{proposition}
	Let $\alpha>1$, $m>0$, and $(\phi_t)$ a parabolic semigroup in $\D$  whose Koenigs domain is $\Pi_{\alpha,m}$. In that case
		\begin{equation}\label{eq:vt_pi}
		v^T_\phi(t)\sim\frac{1}{2\alpha}\log t
		\end{equation}
	and for $\epsilon,\epsilon'>0$, $\epsilon'<m$
		\begin{equation}\label{eq:vo_pi}
		\frac{(m-\epsilon')^{1/\alpha}\alpha}{2(\alpha-1)}\pi\cdot t^{1-1/\alpha}\lesssim v_\phi^o(t)\lesssim\frac{(m+\epsilon)^{1/\alpha}\alpha}{2(\alpha-1)}\pi\cdot t^{1-1/\alpha}.
		\end{equation}
\end{proposition}

Therefore we can also say something in terms of Euclidean distances.
\begin{corollary}
	Let $\alpha>1$ and $m>0$. Given a parabolic semigroup in $\D$ $(\phi_t)$ whose Koenigs domain is $\Pi_{\alpha,m}$, if $\tau\in\de\D$ is the Denjoy-Wolff point of the semigroup, then there exists $R_1=R_1(\alpha,m)>1$ so that
	\begin{align}
	\begin{aligned}\label{eq:stolz_pi}
		\phi_t(0) &\in \mathcal{S}(\tau,R_1\cdot t^{1/\alpha}) &\text{if } &1\le t < R_1^{\alpha}\\
	\phi_t(0) &\in \mathcal{S}(\tau,R_1^2)\setminus([0,1)\tau) &\text{if } &t= R_1^{\alpha}\\
	\phi_t(0) &\in \mathcal{S}(\tau,R_1\cdot t^{1/\alpha})\setminus\overline{\mathcal{S}(\tau,R_1^{-1}\cdot t^{1/\alpha})} &\text{if } &t> R_1^{\alpha}
	\end{aligned}
	\end{align}
	Moreover, let $\Lambda_m^+,\Lambda_m^->0$ such that 
	$\Lambda_m^-<m^{1/\alpha}\alpha/(\alpha-1)< \Lambda_m^+$. Then there exists $R_2=R_2(\alpha,m,\Lambda_m^\pm)>1$ so that
	\begin{equation}\label{eq:rate_pi}
	\frac{1}{R_2}\exp\left(-\Lambda_m^+\pi\cdot t^{1-1/\alpha}\right)\le|\tau-\phi_t(0)|\le R_2\cdot \exp\left(-\Lambda_m^-\pi\cdot t^{1-1/\alpha}\right).
	\end{equation}
\end{corollary}
\begin{proof}
	By \eqref{eq:vt_eu} and \eqref{eq:vt_pi}, there exists $r_1>0$, depending on $\alpha$ and $m$, such that for $t\ge 1$
	\[
	\frac{1}{\alpha}\log t - 3\log 2 - r_1<\log\frac{|\tau-\phi_t(0)|}{1-|\phi_t(0)|}< \frac{1}{\alpha}\log t + 3\log 2 + r_1.
	\]
	By defining $R_1=8e^{r_1}$ and taking exponential, one gets \eqref{eq:stolz_pi}. Indeed, as $|\tau-\phi_t(0)|/(1-|\phi_t(0)|)\ge 1$ with equality holding only when $\phi_t(0)\in [0,1)\tau$, the left inequality is trivial when $t^{1/\alpha}<R_1$ and says that $\phi_t(0)\notin[0,1)\tau$ when $t^{1/\alpha}=R_1$.
	
	Now fix $\epsilon,\epsilon'>0$, $\epsilon'<m$ and set
	\[
	\Lambda_m^-:=\frac{(m-\epsilon')^{1/\alpha}\alpha}{\alpha-1},\quad \Lambda_m^+:=\frac{(m+\epsilon)^{1/\alpha}\alpha}{\alpha-1}.
	\]
	From the relations \eqref{eq:vo_eu} and \eqref{eq:vo_pi} it follows that there exists another constant $r_2>0$, now depending also on the constants above, so that for all $t\ge 1$
	\[
	\Lambda_m^-\pi\cdot t^{1-1/\alpha} - \log 2 - r_2\le\log\frac{1}{|\tau-\phi_t(0)|}\le \Lambda_m^+\pi\cdot t^{1-1/\alpha} + \log 2 + r_2
	\]
	and \eqref{eq:rate_pi} trivially follows by setting $R_2=2e^{r_2}$.
\end{proof}
\begin{remark}
If we consider $\Pi_{\alpha,m}$ with $m>0$ and $\alpha>2$, then the associated semigroup $(\phi_t)$ is such that \[
\sup_{t\ge 1}\,\left|v^T_\phi(t)-\frac{1}{2}\log t\right|=+\infty.
\]
So we have a family of counterexamples of the conjecture in \cite[Section 8, Question 2]{Br}. This question has been recently answered negatively by Zarvalis in \cite{Zar}, where a counterexample is given by removing countable slits from an half-plane (so the boundary is not smooth) and the proof is based on harmonic measures in rectangular domains.
\end{remark}


\end{document}